\documentclass[sn-mathphys-num]{sn-jnl}
\usepackage{enumitem}
\usepackage{amsfonts}
\usepackage[dvipsnames]{xcolor}
\usepackage{mathtools}
\usepackage{stackengine}
\usepackage{verbatim}
\usepackage{tikz-cd} 
\usepackage{float}
\usepackage{hyperref} 
\usepackage{amsmath}
\numberwithin{equation}{section}

\makeatletter
\@namedef{subjclassname@2020}{%
  \textup{2020} MSC}
\makeatother

\def\XXint#1#2#3{{\setbox0=\hbox{$#1{#2#3}{\int}$ }
\vcenter{\hbox{$#2#3$ }}\kern-.6\wd0}}

\makeatletter
\newcommand*{\rom}[1]{\expandafter\@slowromancap\romannumeral #1@}

\newcommand{\prim}{\text{prim}}

\newcommand{\SL}{\mathrm{SL}}

\newcommand{\M}{\mathrm{M}}

\newcommand{\X}{\mathcal{X}}

\newcommand{\Y}{\mathcal{Y}}

\newcommand{\R}{\mathbb{R}}

\newcommand{\e}{\varepsilon}

\newcommand{\tF}{\tilde{F}}
\newcommand{\Pp}{\mathcal{P}}

\newcommand{\z}{\mathcal{Z}}

\newcommand{\Z}{\mathbb{Z}}

\newcommand{\N}{\mathbb{N}}

\newcommand{\bthm}{\begin{thm}}
\newcommand{\ethm}{\end{thm}}
\newcommand{\bproof}{\begin{proof}}
\newcommand{\eproof}{\end{proof}}
\newcommand{\blem}{\begin{lem}}
\newcommand{\elem}{\end{lem}}
\newcommand{\brem}{\begin{rem}}
\newcommand{\erem}{\end{rem}}
\newcommand{\eeqn}{\end{equation}}
\newcommand{\eeqnn}{\end{equation*}}
\newcommand{\beqn}{\begin{equation}}
\newcommand{\beqnn}{\begin{equation*}}
\newcommand{\eprop}{\end{prop}}
\newcommand{\eexm}{\end{exm}}
\newcommand{\enexm}{\end{nexm}}
\newcommand{\ecor}{\end{cor}}
\newcommand{\bcor}{\begin{cor}}
\newcommand{\bexm}{\begin{exm}}
\newcommand{\bnexm}{\begin{nexm}}
\newcommand{\bprop}{\begin{prop}}
\newcommand{\bdefn}{\begin{defn}}
\newcommand{\edefn}{\end{defn}}
\newcommand{\benum}{\begin{enumerate}}
\newcommand{\eenum}{\end{enumerate}}

\newcommand{\cN}{\mathcal{N}}

\newcommand{\Mat}{\M_{m \times n}(\R)}

\newcommand{\supp}{\text{supp}}

\newcommand{\cum}{\text{Cum}}

\title{L\'{e}vy-Khintchine Theorems: effective results and central limit theorems}

\begin{document}
\theoremstyle{plain}
\newtheorem{thm}{Theorem}[section]
\newtheorem{lem}[thm]{Lemma}
\newtheorem{prop}[thm]{Proposition}
\newtheorem{cor}[thm]{Corollary}
\newtheorem{question}{Question}
\newtheorem{con}{Conjecture}
\theoremstyle{definition}
\newtheorem{defn}[thm]{Definition}
\newtheorem{exm}[thm]{Example}
\newtheorem{nexm}[thm]{Non Example}
\newtheorem{prob}[thm]{Problem}

\theoremstyle{remark}
\newtheorem{rem}[thm]{Remark}

\author*[1]{\fnm{Gaurav} \sur{Aggarwal}}\email{gaurav@math.tifr.res.in}
\author[2]{\fnm{Anish} \sur{Ghosh}}\email{ghosh@math.tifr.res.in}

\affil[1]{
  \orgdiv{School of Mathematics}, 
  \orgname{Tata Institute of Fundamental Research}, 
  \orgaddress{\city{Mumbai}, \postcode{400005}, \country{India}}
}

\affil[2]{
  \orgdiv{School of Mathematics}, 
  \orgname{Tata Institute of Fundamental Research}, 
  \orgaddress{\city{Mumbai}, \postcode{400005}, \country{India}}
}

\date{}




\abstract{  
The L\'evy–Khintchine theorem is a classical result in Diophantine approximation that describes the asymptotic growth of the denominators of convergents in the continued fraction expansion of a typical real number. An effective version of this theorem was proved by Phillip and Stackelberg (\textit{Math. Annalen}, 1969) and Central Limit Theorems were proved by several authors \cite{Ibragimov, Misevicius, Morita, Vallee}. In this work, we develop a new approach towards quantifying the L\'evy–Khintchine theorem. Our methods apply to the setting of higher-dimensional simultaneous Diophantine approximation, thereby providing an effective version of a theorem of Cheung and Chevallier (\textit{Annales scientifiques de l'ENS}, 2024). Further, we prove a Central Limit Theorem for best approximations in all dimensions. Unlike previous approaches to the one-dimensional problem, our approach relies on techniques from homogeneous dynamics. 
}

\maketitle


\section{Introduction}
For a real number $\theta$, we denote its continued fraction expansion as usual:
\begin{align*}
    \theta = a_0+ \frac{1}{a_1 + \frac{1}{a_2 + \frac{1}{a_3 + \frac{1}{\ddots}}}}, 
\end{align*}

where

\begin{align*}
    \frac{p_k}{q_k}= a_0+ \frac{1}{a_1 + \frac{1}{a_2 + \frac{1}{ \ddots + \frac{1}{a_k}}}}
\end{align*}

\noindent denotes the $k$th convergent. A classical and beautiful theorem of Khintchine \cite{Khintchine36} states that for Lebesgue a.e. $\theta$:
\begin{equation}\label{LK}
\lim_{n \rightarrow \infty} \frac{\log q_k}{k} = \frac{\pi^2}{12 \log 2}.
\end{equation}

More precisely, Khintchine established that the limit was constant almost surely, while the value of the constant was computed by L\'{e}vy \cite{Levy36}. This theorem is therefore called the L\'{e}vy-Khintchine theorem.  Subsequently, this theme has been extensively studied. In a probabilistic direction, central limit theorems were established by Ibragimov \cite{Ibragimov}, Misevicius \cite{Misevicius} and by Morita \cite{Morita} and Vall\'{e}e \cite{Vallee}, and a law of iterated logarithm, which effectivises the L\'evy–Khintchine theorem was proved by Philipp and Stackelberg \cite{PhilippStackelberg}, following earlier results by LeVeque \cite{LeVeque58} and Philipp \cite{Philipp67, PhilippDuke}. Large deviation results were obtained in \cite{Takahasi}. There has been significant work on characterising numbers via their L\'{e}vy constants, see \cite{PollicottWeiss},  \cite{BugeaudKimLee} and the references therein.\\

An important recent development is the generalization of the L\'evy-Khintchine theorem to higher dimensional simultaneous Diophantine approximation by Cheung and Chevallier \cite{CC19}. Subsequently, the present authors resolved Cheung's conjecture and proved several related results in \cite{AG24Levy}. However, all the higher dimensional work thus far has been qualitative. This bring up the following fundamental questions:\\

\noindent \textbf{Is it possible to effectivise the Cheung-Chevallier theorem?}\\ 

\noindent \textbf{Can one obtain central limit theorems for best approximations in simultaneous Diophantine approximation?}\\

\noindent In this paper, we answer both these fundamental questions in the affirmative. We provide an \emph{effective} L\'evy-Khintchine theorem in simultaneous approximation. We also prove central limit theorems for the growth of best approximations. In a companion paper \cite{aggarwalghosh25}, we explore quantitative refinements of Doeblin-Lenstra theorems.\\ 

We begin by introducing some notation.\\

\subsection{Statement of results in higher dimension}
Fix $m,n$ to be positive integers and norms $\|.\|_{\R^m}$ on $\R^m$ and $\|.\|_{\R^n}$ on $\R^n$. An integer vector $(p,q) \in \Z^m \times (\Z^n \setminus \{0\})$ is called a best approximation of a matrix $\theta \in \Mat$, corresponding to norms $\|.\|_{\R^m}$ and $\|.\|_{\R^n}$, if there is no integer solution $(p',q') \in \Z^m \times (\Z^n \setminus \{0\})$ other than $(p,q)$ and $(-p,-q)$ to the following inequalities
    \begin{align*}
        \|p'+ \theta q'\|_{\R^m} &\leq \|p+ \theta q\|_{\R^m}, \\
        \|q'\|_{\R^n} & \leq \|q\|_{\R^n}.
    \end{align*}
    We define $\cN(\theta, T)$ as the number of best approximations $(p,q) \in \Z^m \times (\Z^n \setminus \{0\})$ of $\theta$ such that $\|q\|_{\R^n} \leq e^T$.

    Note that for a fixed $\theta$, if $(p_j,q_j)_{j \in \N} \in \Z^m \times (\Z^n \setminus \{0\})$ denote the best approximates of $\theta$, then for $T= \log \|q_k\|$, we have
        \begin{align*}
            \cN(\theta, T)= k,
        \end{align*}
        since $\{(p_1, q_1), \ldots, (p_k,q_k)\}$ are exactly the best approximates $(p,q)$ of $\theta$ with $\|q\| \leq \|q_k\|= e^T $. This means that
        \begin{align*}
            \lim_{k \rightarrow \infty} \frac{\log\|q_k\|}{k}=  \lim_{T \rightarrow \infty}\frac{T}{ \cN(\theta, T)},
        \end{align*}
        provided the latter exists. In other words, the L\'evy-Khintchine theorem (Theorem 1 in \cite{CC19}) is a result about the asymptotic number of best approximates, namely
        \begin{align}
        \label{eq: levy khintchin as counting}
            \cN(\theta, T) \sim  \gamma T.
        \end{align}

Our first theorem establishes an explicit rate in equation \eqref{eq: levy khintchin as counting}, thereby making the L\'evy–Khintchine theorem effective. 
    \begin{thm}
    \label{thm: main effective levy}
    There exists  $\gamma>0$ such that the following holds. For any $\varepsilon > 0$, we have
    \begin{align}
        \label{eq: thm: main effective levy}
        \cN(\theta, T) = \gamma T + O_{\e, \theta}\left(T^{1/2} \log^{\frac{3}{2}+\varepsilon} T\right),
    \end{align}
    for Lebesgue almost every $\theta \in \Mat$.
\end{thm}
\vspace{0.2em}

    The second main theorem of this paper is a Central Limit Theorem for best approximations.
    \begin{thm}
    \label{thm: CLT general}
    Let $\mu$ denote the probability measure on $\M_{m \times n}([0,1])$ obtained by restriction of Lebesgue measure. Then there exists a $\sigma > 0$ such that for every $\xi \in \mathbb{R}$,
    $$
    \mu\left( \left\{ \theta \in \Mat : \frac{\cN(\theta, T) - \gamma T}{T^{1/2}} < \xi \right\} \right) \longrightarrow \mathrm{Norm}_{\sigma}(\xi),
    $$
    as $T \to \infty$, where
    $$
    \mathrm{Norm}_{\sigma}(\xi) := \frac{1}{\sigma \sqrt{2\pi}} \int_{-\infty}^\xi e^{-\frac{1}{2} \left( \frac{x}{\sigma} \right)^2} \, dx,
    $$
    and $\gamma$ is as in Theorem \ref{thm: main effective levy}.
\end{thm}

\vspace{0.2em}
\begin{rem}
In the case $m = n = 1$, the error term obtained by Philipp and Stackelberg \cite{PhilippStackelberg} is better than the one we obtain in Theorem \ref{thm: main effective levy}. Namely, they obtain an error $O(\sqrt{T\log\log T})$. Also, the one-dimensional central limit theorems cf. \cite{Misevicius, Morita, Vallee} come with a rate of convergence.
\end{rem} \vspace{0.2em}
    
\begin{rem} 
If we consider simultaneous ``$\varepsilon$" approximations in place of best approximations, then similar counting results have been known  by classic work of Wolfgang Schmidt \cite{S60} (see also \cite{WYYK}) but with a worse error term, namely $O(T^{1/2}\log^{2+\varepsilon}T)$. Also, the Central Limit Theorems in this context have been established by Dolgopyat, Fayad and Vinogradov \cite{DFV} as well as Bj\"{o}rklund and Gorodnik \cite{BG}, see also \cite{AG23}. We refer the reader also to \cite{DFL22} for a wealth of related results.

However, we emphasize that Theorems~\ref{thm: main effective levy} and~\ref{thm: CLT general} address entirely different questions. In particular, they neither follow from any of these works, nor have any implication toward them. 
\end{rem}\vspace{0.2em}

The main novelty of this paper lies in establishing \emph{quantitatively stronger effective results} for the L\'evy–Khintchine theorem than those obtained in the qualitative frameworks of \cite{SW22,CC19,AG24Levy}. These results are achieved through arguments that rely on tools significantly simpler than the heavy cross-section machinery developed in those earlier works.  

At a conceptual level, the key insight of the paper is that the L\'evy–Khintchine theorem can be viewed as describing the convergence of Birkhoff sums of a certain function \( f \) to its space average, under an ergodic transformation on the space of lattices, for almost every point of the relevant submanifold. This reinterpretation offers a more flexible and powerful framework, allowing the classical (effective) L\'evy–Khintchine theorem—previously established only for almost every point on the real line—to be extended to far more delicate settings, such as almost every point on the middle-third Cantor set, and more generally, on self-similar fractals. Such extensions appear to be beyond reach using the cross-section techniques of \cite{SW22,CC19,AG24Levy} or the continued fraction methods of \cite{PhilippStackelberg,Misevicius,Morita,Vallee}. The study of Diophantine approximation on the middle-third Cantor set, which motivates this direction, stems from a question of Mahler~\cite{Mahler}. We anticipate that this new ergodic-theoretic perspective will yield further advances, including laws of large deviations and other probabilistic limit theorems, extending both the L\'evy–Khintchine framework and the scope of Mahler’s problem. We refer the reader to Theorem \ref{thm: main effective levy general} and Remark~\ref{rem: Cantor} in Section~\ref{sec: Main results} below. 

Once the reduction to Birkhoff sums is established, our strategy for obtaining effective estimates follows the general spirit of \cite{KSW}, while the arguments related to the central limit theorem are inspired by \cite{BG}. However, a direct adaptation of these results is not possible, since in our setting the observable \( f \) is \emph{discontinuous} and the underlying measure may fail to be absolutely continuous with respect to Lebesgue measure. The main technical contribution of this paper is to overcome these obstacles by analyzing the \emph{average continuity} properties of \( f \)—see Lemma~\ref{lem: perturbed difference},~\ref{lem: approximation f error}. We show that although \( f \) is discontinuous, it exhibits continuity on average, in the sense that  
\[
\int |f(a_t u(\theta)\Gamma) - f(g a_t u(\theta)\Gamma)| \, d\mu(\theta) \to 0 
\quad \text{as } g \to e \text{ and } t \to \infty,
\]
provided \( t > -c_0 \log d(g,e) \) for some constant \( c_0 \), where \( d(\cdot,\cdot) \) denotes the distance on \( \SL_{m+n}(\R) \). The proof of this average continuity property forms the technical core of the paper.  

Building on this foundation, we prove Theorem~\ref{thm: main abstract theorem}, extending the results of \cite{KSW,BG} by establishing effective estimates and a central limit theorem for Birkhoff sums of discontinuous functions and measures that may be singular with respect to Lebesgue measure. The crucial assumptions are the average continuity of \( f \) (as above) and the validity of the condition (EMEI), see Section~\ref{sec: Main results} for the definition.  

While the overall structure of the proof of Theorem~\ref{thm: main abstract theorem} follows the general approach of \cite{KSW,BG}, two substantive modifications are required. First, we control the error terms arising from the convolution kernels used to regularize \( f \). Second, we identify that the essential ingredient in \cite{KSW,BG} is not the absolute continuity of the measure, but rather the (EMEI) condition. Aside from these innovations, our argument shares several technical features with the previous works.

In summary, the present work advances the theory in two principal directions. It introduces a new ergodic-theoretic framework that enables effective L\'evy–Khintchine results to be established on fractal sets, and it extends the methods of \cite{KSW,BG} to encompass discontinuous observables and singular measures under the (EMEI) condition.

\section{Main Results}
\label{sec: Main results}
 Let $G= \SL_{m+n}(\R)$ and $\Gamma = \SL_{m+n}(\Z)$. Let $d(\cdot, \cdot)$ denote a right invariant distance function on $G$, and let $m_G$ denote the Haar measure on $G$ normalised so that the fundamental domain of the $\Gamma$ action on $G$ has measure equal to $1$. We denote by $\X $ the homogeneous space $G/\Gamma$ which, as is well known, can be identified with the space of unimodular lattices in $\R^{m+n}$ via the identification 
    $$
    A\Gamma \mapsto A\Z^{m+n}.
    $$
    Let $\mu_\X$ denote the unique $G$-invariant probability measure on $\X$. \\
    
    For $\theta \in \Mat$ and $t \in \R$, let us define $ u(\theta), a_t \in G$ as
    \begin{align}
        u(\theta) = \begin{pmatrix}
            I_m & \theta \\ & I_n
        \end{pmatrix}, \quad  a_t = \begin{pmatrix}
            e^{\frac{n}{m}t} I_m \\ & e^{-t} I_n
        \end{pmatrix}.
    \end{align}

\begin{defn}
    A probability measure $\mu$ on $\Mat$ is said to satisfy Condition (EMEI) (short for \emph{Effective Multi-equidistribution for the Identity coset under the diagonal flow $a_t$}) if it satisfies the following properties:
    \begin{itemize}
        \item $\mu$ is compactly supported,
        \item there exists a $k \in \N$ such that for all $r \in \N$, there exists a $ \delta_r>0$ such that the following holds: for all $F_0 \in C^{\infty}(\Mat )$, $F_1, \ldots, F_r \in C_c^{\infty}(\X)$ and $t_1, \ldots,t_r >0 $, we have 
    \begin{align}
         \int_{\Mat} F_0(\theta) \left( \prod_{i=1}^r F_i(a_{t_i} u(\theta) \Gamma)\right) \, d\mu(\theta) &=  \mu(F_0) \mu_\X( F_1) \cdots \mu_\X(F_r) \nonumber\\
         & \quad + {O}_{ r} \left(e^{-\delta_r D(t_1, \ldots, t_r)} \|F_0\|_{C^k}  \prod_{i=1}^r \|F_i\|_{C^k} \right),  \label{eq: mix hom identity}
    \end{align}
    where $D(t_1, \ldots, t_r) = \min\{t_i, |t_i - t_j|: 1\leq i ,j \leq r, i \neq j\}$.
    \end{itemize}
\end{defn}

\begin{rem}
\label{rem: class of measures satifying EMEI}
    It is well known that if $\mu$ equals the restriction of the Lebesgue measure to $\M_{m \times n}([0,1])$, then $\mu$ satisfies the condition \textnormal{(EMEI)}, see \cite[Cor.~3.5]{KM3}. In fact, it is striking that even measures singular to Lebesgue, such as the Hausdorff measure supported on the middle-third Cantor set, can be shown to be \textnormal{(EMEI)}. This is proved in our paper \cite{aggarwalghosh25}.
\end{rem}

The main result of the paper is the following general Theorem from which Theorems \ref{thm: main effective levy} and \ref{thm: CLT general} follow immediately. 
    \begin{thm}
    \label{thm: main effective levy general}
    Fix $\mu$ on $\Mat$ satisfying condition (EMEI). Then there exists  $\gamma, \sigma>0$ (depending only on $m,n$ and choice of norms, and independent of $\mu$) such that the following holds.
    \begin{itemize}
        \item[(i)] For any $\varepsilon > 0$ and for $\mu$-almost every $\theta \in \Mat$, we have
    \begin{align}
        \label{eq: thm: main effective levy general}
        \cN(\theta, T) = \gamma T + O_{\e, \theta}\left(T^{1/2} \log^{\frac{3}{2}+\varepsilon} T\right).
    \end{align}
    \item[(ii)] For every $\xi \in \mathbb{R}$,
    $$
    \mu\left( \left\{ \theta \in \Mat : \frac{\cN(\theta, T) - \gamma T}{T^{1/2}} < \xi \right\} \right) \longrightarrow \mathrm{Norm}_{\sigma}(\xi),
    $$
    as $T \to \infty$, where
    $$
    \mathrm{Norm}_{\sigma}(\xi) := \frac{1}{\sigma \sqrt{2\pi}} \int_{-\infty}^\xi e^{-\frac{1}{2} \left( \frac{x}{\sigma} \right)^2} \, dx.
    $$
    \end{itemize}
\end{thm}
\vspace{0.2in}

\begin{rem}
    The constants \(\gamma, \sigma\) above, depend only on \(m\), \(n\) and the chosen norms on $\R^m$ and $\R^n$. In particular, these are given by the integrals in \eqref{eq: def gamma intro} and \eqref{eq: def sigma} where the function $f$ is defined by \eqref{eq: def f}. Our proof does not yield explicit values for $\gamma$ and $\sigma$. Indeed, it would be an interesting, albeit challenging problem to estimate these integrals, but we do not pursue it here. Very few results are known in this direction. In the case \(m = n = 1\), we have \(\gamma = \frac{24 \log 2}{\pi^2}\). In higher dimensions and for standard Euclidean norms, \(\gamma = \frac{2}{L_{m,n}}\), where \(L_{m,n}\) is the constant appearing in Theorem~1 of Cheung and Chevallier~\cite{CC19}. In both cases, the factor of \(2\) arises because we count with sign. Apart from dimension one, the only other known information about $\gamma$ is obtained from a separate paper of Cheung and Chevallier \cite{cheung2021valuedimensionallevysconstant} where they estimate $L_{2,1}$ using numerical integration and therefore obtain an estimate for $\gamma$. 
\end{rem}\vspace{0.2em}

\begin{rem}
\label{rem: Cantor}
    By \cite[Theorem~1.2]{BHZ2025} and \cite[Theorem~2.3]{aggarwalghosh25}, any non-atomic self-similar probability measure on \( M_{m \times 1}(\R) \) generated by similarities with a common contraction ratio satisfies condition (EMEI). Thus, Theorem~\ref{thm: main effective levy general} implies that Theorems~\ref{thm: main effective levy} and~\ref{thm: CLT general} remain valid when \(\mu\) is replaced by a broad class of fractal measures, including, in particular, the \({\log 2}/{\log 3}\)-dimensional Hausdorff measure supported on the middle-third Cantor set.
\end{rem}


\section{The function $f$}
    Throughout the paper, we fix norms $\|\cdot\|_m$ on $\mathbb{R}^m$ and $\|\cdot\|_n$ on $\mathbb{R}^n$. By rescaling these norms if necessary, we may assume that the unit balls $\{ x \in \mathbb{R}^m : \|x\|_m \leq 1 \}$ and $\{ x \in \mathbb{R}^n : \|x\|_n \leq 1 \}$ have Lebesgue measure greater than $2^m$ and $2^n$, respectively. For simplicity, we will use the same notation $\|\cdot\|$ for both norms $\|\cdot\|_m$ and $\|\cdot\|_n$.

Let $\pi_1: \mathbb{R}^{m+n} = \mathbb{R}^m \times \mathbb{R}^n \to \mathbb{R}^m$ and $\pi_2: \mathbb{R}^{m+n} = \mathbb{R}^m \times \mathbb{R}^n \to \mathbb{R}^n$ denote the natural projections. We define a norm on $\mathbb{R}^{m+n}$, also denoted by $\|\cdot\|$, as
\[
\|v\| := \max\left\{ \|\pi_1(v)\|,\, \|\pi_2(v)\| \right\} \quad \text{for } v \in \mathbb{R}^{m+n}.
\]
Given $v \in \mathbb{R}^{m+n}$, define the box
\begin{align}
    \label{eq: def C_v}
    C_v := \left\{ (x, y) \in \mathbb{R}^m \times \mathbb{R}^n : \|x\| \leq \|\pi_1(v)\|,\, \|y\| \leq \|\pi_2(v)\| \right\}.
\end{align}
We now define the function $f : \X \to \mathbb{Z}_{\geq 0}$ by
\begin{align}
    \label{eq: def f}
    f(\Lambda) := \#\left\{ v \in \Lambda :  \Lambda_\prim \cap C_v  = \{ \pm v \},\, 1 \leq \|\pi_2(v)\| < e,\, \|\pi_1(v)\| \leq 1 \right\}.
\end{align}
Note that if $v \in \Lambda \in \X$ satisfies $1 \leq \|\pi_2(v)\| < e$ and $\#\left( \Lambda_\prim \cap C_v \right) = 0$, then by Minkowski’s theorem, it must follow that $\|\pi_1(v)\| \leq 1$. Hence, the last condition in the definition of $f$ is, in fact, redundant.

The following lemma highlights the significance of $f$ by showing that the number of best approximations of $\theta$ is equal to the Birkhoff sum of $f$ for the action of $a_1$ starting from $u(\theta)\Gamma$. 
\begin{lem}
    \label{lem: correpondence with best approx}
    For all $T \geq 0$, we have
    \begin{align}
      \sum_{i=0}^{\lfloor T \rfloor-1 } f(a_i u(\theta) \Gamma) \leq \cN(\theta,T) \leq \sum_{i=0}^{\lfloor T \rfloor } f(a_i u(\theta) \Gamma),
    \end{align}
    where $\lfloor T \rfloor$ denotes the smallest integer less than or equal to $T$.
\end{lem}
\begin{proof}
It enough to show that the number of best approximations $(p,q)$ of $\theta$ with $e^M \leq \|q\|< e^{M+1}$ equals $f(a_M u(\theta) \Gamma)$ for all $M \in \Z_{\geq 0}$. To show this, we let $\text{Best}(\theta, a,b)$ denote the set of all best approximations of $\theta$ with $e^a \leq \|q\| < e^b$. For any lattice $\Lambda$, we define $S_{\Lambda}$ as the set of all $v \in \Lambda$ such that
        \begin{align*}
            \Lambda_\prim \cap C_v= \{\pm v\}, \quad \quad 1 \leq \|\pi_2(v)\| < e, \quad \quad \|\pi_1(v)\| \leq 1.
        \end{align*}
     Fix $M \in \N$. Note that $(p,q) \in \text{Best}(\theta, M, M+1)$ if and only if $e^M\leq \|q\| < e^{M+1}$ and there is no integer solution $(p',q') \in \Z^m \times (\Z^n \setminus \{0\})$ other than $(p,q)$ and $(-p,-q)$ to the following inequalities
    \begin{align*}
        \|p'+ \theta q'\|_{\R^m} &\leq \|p+ \theta q\|_{\R^m}, \\
        \|q'\|_{\R^n} & \leq \|q\|_{\R^n}.
    \end{align*}
    The latter holds if and only if $e^M\leq \|q\| < e^{M+1}$ and there are no primitive vectors of the lattice $u(\theta)\Z^{m+n}$ in the region $C_{(p+\theta q, q)}$, other than $\pm (p+\theta q, q)$. This is further equivalent to the condition that $1 \leq \|e^{-M} q\|= \|\pi_2(a_M(p+\theta q, q))\| < e$ and there are no primitive vectors of the lattice $a_M u(\theta) \Z^{m+n}$ in the region $a_M C_{(p+\theta q, q)} = C_{a_M(p+\theta q, q)}$, other than $\pm a_M(p+\theta q, q)$. This holds if and only if $a_M(p+\theta q, q) \in S_{a_Mu(\theta)\Gamma}$. This gives a one-to-one correspondence between $\text{Best}(\theta, M, M+1)$ and elements of $S_{a_Mu(\theta)\Gamma}$. Hence, the number of best approximations $(p,q)$ of $\theta$ with $e^M \leq \|q\|< e^{M+1}$ equals $f(a_M u(\theta) \Gamma)$ for all $M \in \Z_{\geq 0}$. The lemma now follows.
\end{proof}

\section{Properties of $f$}

In this section, we will show that $f$ is a bounded function on $\X$. Additionally, we will also study the continuity properties of $f$.

\subsection{Boundedness of $f$}
\begin{lem}
        \label{lem: bounded f}
        There exists $M \geq 1$ such that $f(\Lambda) \leq M$ for all $\Lambda \in \X$.
\end{lem}
 \begin{rem}
    Note that the value of $M$ depends on $m,n$ as well as on the choice of norms on $\R^m$ and $\R^n$.
\end{rem}
\begin{proof}[Proof of Lemma~\ref{lem: bounded f}]
    The proof is motivated from the proof of \cite[Prop.~9.8]{SW22}. Let $M$ be large enough so that any subset of $B_1 \times B_e$ of cardinality $M+1$ contains distinct points $x, y$ satisfying
\[
\|\pi_1(x-y)\| \leq \frac{1}{3} \quad \text{and} \quad \|\pi_2(x-y)\| < 1.
\]
Such a number $M$ exists by compactness of $B_1 \times B_e$.  
Moreover, by applying a linear transformation that dilates the horizontal subspace $\mathbb{R}^m \times \{0\}$, one sees that for any $r > 0$, in any subset of $B_r \times B_e$ of cardinality $M+1$, there exist distinct points $x, y$ such that
\[
\|\pi_1(x-y)\| \leq \frac{r}{3} \quad \text{and} \quad \|\pi_2(x-y)\| < 1.
\]

We claim that $f(\Lambda) \leq M$ for all $\Lambda$.  
Suppose by contradiction that there exist vectors $v_1, \ldots, v_{M+1} \in (B_1 \times B_e) \cap \Lambda$ such that
\[
\Lambda_\prim \cap C_{v_j} = \{\pm v_j\} \quad \text{for all} \quad j=1, \ldots, M+1.
\]
By reordering the $v_j$, we may assume that
\[
\|\pi_2(v_1)\| \leq \|\pi_2(v_2)\| \leq \cdots \leq \|\pi_2(v_{M+1})\|.
\]
Then, by the condition that \(\Lambda_\prim \cap C_{v_j} = \{\pm v_j\}\), it follows that
\[
\|\pi_1(v_1)\| \geq \|\pi_1(v_2)\| \geq \cdots \geq \|\pi_1(v_{M+1})\|.
\]
Thus, the points $v_1, \ldots, v_{M+1}$ are distinct and lie in $B_{\|\pi_1(v_1)\|} \times B_e$.  
By the choice of $M$, there exist indices $i \neq j$ such that
\[
\|\pi_1(v_i - v_j)\| < \frac{\|\pi_1(v_1)\|}{3} \quad \text{and} \quad \|\pi_2(v_i - v_j)\| < 1.
\]
Thus, the nonzero vector $v_i - v_j$ lies in $\Lambda$, is distinct from $\pm v_1$, and belongs to $C_{v_1}$.  
This contradicts the assumption that
\[
\Lambda_\prim \cap C_{v_1} = \{\pm v_1\}.
\]
Hence, the lemma follows.
\end{proof}

\subsection{Continuity of $f$}

Since $f$ is a non-constant function taking values in $\Z$, it is clear that $f$ is a discontinuous function. However, it is still approximately continuous on average. This subsection aims to prove the same. To study the continuity properties, we introduce the following auxiliary functions.
\vspace{0.5cm}

For all sufficiently small $\e > 0$, let $\varphi_\e : \mathbb{R}^{m+n} \to \mathbb{R}$ be a continuous function such that
\begin{align}
    \label{eq: def varphi e}
    \varphi_\e(v) := 
    \begin{cases}
        1 & \text{if } v \in \{ (x, y) : \|x\| \leq 1+\e,\, 1-\e \leq \|y\| \leq e+\e \} \\ 
        &\quad \text{and } v \notin \{ (x, y) : \|x\| < 1-\e,\, 1+\e < \|y\| < e-\e \}, \\
        0 & \text{otherwise}.
    \end{cases}
\end{align}
Note that by the Siegel mean value theorem \cite{sie}, we have
\begin{align}
    \int_{\X} \sum_{v \in \Lambda_\prim} \varphi_\e(v)\, d\mu_\X(\Lambda) 
    &= \frac{1}{\zeta(m+n)} \int_{\mathbb{R}^{m+n}} \varphi_\e(v)\, dv \nonumber \\
  &\ll \Bigl( (1+\e)^m \bigl( (e+\e)^n - (1-\e)^n \bigr) \Bigr)- \Bigl( (1-\e)^m \bigl( (e-\e)^n- (1+\e)^n \bigr) \Bigr) \nonumber\\ 
    &\ll \e, \label{eq: int varphi}
\end{align}
where the implied constant can be chosen independently of $\e$, say $C_\varphi$, and depend only on $m$, $n$ and the choice of norms on $\R^m$ and $\R^n$.
\vspace{0.5cm}

Similarly, for small $\e > 0$, define $\Phi_\e : \mathbb{R}^{m+n} \times \mathbb{R}^{m+n} \to \mathbb{R}$ by
\begin{align}
    \label{eq: def Phi e}
    \Phi_\e(v, w) := 
    \begin{cases}
        1 & \text{if } v, w \in \{ (x, y) : \|x\| \leq 1+\e,\,  \|y\| \leq e+\e \}, \\
        &\quad \text{and either } |\|\pi_1(v)\| - \|\pi_1(w)\|| \leq \e \text{ or } |\|\pi_2(v)\| - \|\pi_2(w)\|| \leq \e, \\
        0 & \text{otherwise}.
    \end{cases}
\end{align}
Applying Siegel's mean value theorem \cite{sie} and Rogers' formula \cite[Thm.~5]{rog}, we obtain, for all sufficiently small $\e > 0$,
\begin{align}
    &\int_{\X} \sum_{\substack{v, w \in \Lambda_\prim \\ w \neq \pm v}} \Phi_\e(v, w)\, d\mu_\X(\Lambda) \nonumber\\&= \int_{\X} \sum_{v, w \in \Lambda_\prim } \Phi_\e(v, w)\, d\mu_\X(\Lambda) -\int_{\X} \sum_{v \in \Lambda_\prim } \left( \Phi_\e(v,v) + \Phi_\e(v,-v) \right) \, d\mu_\X(\Lambda) \nonumber \\
    &= \frac{1}{\zeta(m+n)^2} \int_{\mathbb{R}^{m+n} \times \mathbb{R}^{m+n}} \Phi_\e(v, w)\, dv\, dw 
    \ll \e, \label{eq: int PHI}
\end{align}
where again, the implied constant can be chosen independently of $\e$, say $C_\Phi$, and depend only on $m$, $n$ and the choice of norms on $\R^m$ and $\R^n$.

\begin{rem}
To justify the estimate $\int_{\R^{m+n}\times\R^{m+n}} \Phi_\e(v,w)\,dv\,dw \ll \e$ in \eqref{eq: int PHI},
fix $v$ in the region $\{ (x, y) : \|x\| \leq 1+\e,\,  \|y\| \leq e+\e \}$. Then $\Phi_\e(v,w)=1$ implies that $w$ lies in
$\{(x,y): \|x\|\le 1+\e,\ \|y\|\le e+\e\}$ and, for some $j\in\{1,2\}$,
\[
\bigl|\|\pi_j(w)\|-\|\pi_j(v)\|\bigr|\le \e.
\]
For each fixed $j$, this restricts $\pi_j(w)$ to an $\e$-thick shell in $\R^{m}$ (if $j=1$)
or in $\R^{n}$ (if $j=2$), whose Lebesgue measure is $O(\e)$, while the other coordinate
remains in a bounded set of $O(1)$-volume. Hence the set of such $w$ has volume $O(\e)$, and Fubini gives $\int \Phi_\e \ll \e$.
\end{rem}
\vspace{0.5cm}

 The following is the main result of this subsection and is the first key ingredient for the proof of the main Theorem.
    \begin{lem}
        \label{lem: perturbed difference}
        There exists $C >0$ such that for all small enough $\e>0$ and $g$ in an $\e$-ball around identity in $G$, we have
        \begin{align}
            \label{eq: lem: perturbed difference 1}
            |f(g\Lambda)- f(\Lambda)| \leq \sum_{v\in \Lambda_\prim } \varphi_{C\e}(v)+   \sum_{\substack{v,w \in \Lambda_\prim \\ w \neq \pm v}} \Phi_{C\e}(v,w),
        \end{align}
        where $\varphi_{C\e}$ and $\Phi_{C\e}$ are defined as in \eqref{eq: def varphi e} and \eqref{eq: def Phi e}.
    \end{lem}
    \begin{proof}
    Fix $C_1>0$ so that for all small enough $\e>0$ and $g$ in an $\e$-ball around the identity element in $G$, we have that 
        \begin{align}
            \label{eq: w 1}
            \|gv-v\| \leq C_1 \e \|v\| , \quad \text{ for all } v \in \R^{m+n}.
        \end{align}
        Let $C= 2 C_1 e$, fix $\e>0$ and $g$ in the $\e$-ball around the origin. For $\Lambda \in \X$, we define $S_\Lambda$ as set of all $v \in \Lambda$ such that
        \begin{align*}
            \Lambda_\prim \cap C_v= \{\pm v\}, \quad \quad 1 \leq \|\pi_2(v)\| < e, \quad \quad \|\pi_1(v)\| \leq 1.
        \end{align*}
        Clearly then we have 
        $$
        |f(g\Lambda)- f(\Lambda)| \leq  \#(g^{-1}S_{g\Lambda} \setminus S_\Lambda) + \#(S_\Lambda \setminus g^{-1}S_{g\Lambda}). 
        $$
        Fix $\Lambda \in \X$ and let 
        \begin{align*}
            g^{-1}S_{g\Lambda} \setminus S_\Lambda &= \{v_1, \ldots, v_p\} \subset \Lambda \\
            S_\Lambda \setminus g^{-1}S_{g\Lambda} &= \{w_1, \ldots, w_q\} \subset \Lambda.
        \end{align*}
        Divide the set $\{v_1, \ldots, v_p\}$ into two parts according to whether or not it belongs to the set $\{(x,y): \|x\| \leq 1, 1 \leq \|y\| < e\}$. Assume that $\{v_1, \ldots, v_r\} \subset \{(x,y): \|x\| \leq 1, 1 \leq \|y\| < e\}$ and $\{v_{r+1}, \ldots, v_p\} \cap\{(x,y): \|x\| \leq 1, 1 \leq \|y\| < e\} = \emptyset$. Also divide the set $\{w_1, \ldots, w_q\}$ into two parts depending on whether it belongs to the set $$\{(x,y): \|x\| < 1-C\e, 1+ C\e < \|y\| < e- C\e\}$$ or not. Assume that $\{w_1, \ldots, w_s\} \subset \{(x,y): \|x\| < 1-C\e, 1+ C\e < \|y\| < e -C\e\}$ and $\{w_{s+1}, \ldots, w_q\} \cap\{(x,y): \|x\| < 1-C\e, 1+ C\e < \|y\| < e -C\e\} = \emptyset$.\\

        We now make the following observations before proceeding. \vspace{0.5cm}

        \noindent {\bf Observation 1:} $\varphi_{C\e}( v_i) =1$ for $i=r+1, \ldots, p$.

        \noindent{\bf Explanation:} Note that since $gv_i \in S_{g\Lambda}$ for $i=r+1, \ldots, p$, we get that $g v_i \in \{(x,y): \|x\| \leq 1, 1 \leq \|y\| \leq e\}$. Using \eqref{eq: w 1}, we get that $v_i = g^{-1}(gv_i) \subset \{(x,y): \|x\| \leq 1+C_1 e \e, 1-C_1 e \e \leq \|y\| \leq e +C_1 e \e\}$. Since we already know that $v_i \notin \{(x,y): \|x\| \leq 1, 1 \leq \|y\| < e\} = \emptyset$, we get that $\varphi_{C\e}( v_i) =1$ for $i=r+1, \ldots, p$. \vspace{0.5cm}

        \noindent {\bf Observation 2:} For $1 \leq i \leq r$, there exists $v_i' \in \Lambda_{\prim} \setminus \{ \pm v_i \}$ such that $\Phi_{C\e}(v_i,v_i') =1$.

        \noindent{\bf Explanation:} Since $v_i \in \{(x,y): \|x\| \leq 1, 1 \leq \|y\| < e\}$ and $v_i \notin S_\Lambda$, we get $v_i' \in \Lambda_\prim$ with $v_i' \neq \pm v_i$ such that $v_i' \in C_{v_i} \subset \{(x,y): \|x\| \leq 1,  \|y\| < e\}$. Thus we have 
        \begin{align}
        \label{eq: w w 1}
            \|\pi_1(v_i')\| \leq \|\pi_1(v_i)\|, \quad \|\pi_2(v_i')\| \leq \|\pi_2(v_i)\|.
        \end{align}
       But since $gv_i' \notin C_{gv_i}$, there exits $j \in \{1,2\}$ such that 
       \begin{align}
       \label{eq: w w 2}
           \|\pi_j(gv_i')\| > \|\pi_j(gv_i)\|.
       \end{align}
       Fix such a $j$. Note that 
        \begin{align}
        \label{eq: w w 3}
           | \|\pi_j(g v_i')\| - \|\pi_j( v_i')\|| \leq C_1 e\e, \quad | \|\pi_j(g v_i)\| - \|\pi_j( v_i)\|| \leq C_1 e \e.
        \end{align}
        Thus we have 
        \begin{align*}
            |\|\pi_j(v_i)\|- \|\pi_j(v_i')\|| &= \|\pi_j(v_i)\|- \|\pi_j(v_i')\|  \quad \text{ using \eqref{eq: w w 1}}\\
            &\leq \|\pi_j(g v_i)\|- \|\pi_j(g v_i')\| + 2C_1 e \e  \quad \text{ using \eqref{eq: w w 3}} \\
            &\leq 2C_1 e\e = C\e \quad \text{ using \eqref{eq: w w 2}.}
        \end{align*}
        Hence the following holds: $v_i,v_i' \in \{(x,y): \|x\| \leq 1,  \|y\| < e\}$ and for at least one $j\in \{1,2\}$, we have $|\|\pi_j(v_i)\|- \|\pi_j(v_i')\|| \leq C \e $. Hence $\Phi_{C\e}(v_i,v_i') =1$. This proves the observation. \vspace{0.5cm}

    \noindent {\bf Observation 3:} $\varphi_{C\e}( w_i) =1$ for $i=s+1, \ldots, q$.

        \noindent{\bf Explanation:} Note that for each $i=s+1, \ldots, q$, since $w_i \in S_{\Lambda}$ we get that $ w_i \in \{(x,y): \|x\| \leq 1, 1 \leq \|y\| \leq e\}$. Also by given assumption, we have  $w_i \notin \{(x,y): \|x\| < 1-C\e, 1+ C\e < \|y\| < e -C\e\}$. Thus we have $\varphi_{C\e}( w_i) =1$ for $i=s+1, \ldots, q$. \vspace{0.5cm}

        \noindent{\bf Observation 4:} For $1 \leq i \leq s$, there exits $w_i' \in \Lambda_{\prim} \setminus \{ \pm w_i \}$ such that $\Phi_{C\e}(w_i,w_i') =1$.

        \noindent{\bf Explanation:} Note that 
        \begin{align*}
            \|gw_i - w_i\| \leq C_1 \e \|w_i\| \leq C_1 \e e.
        \end{align*}
        Thus $gw_i \in \{(x,y): \|x\| \leq 1, 1 \leq \|y\| < e\}$. But $gw_i \notin S_{g\Lambda}$. So there exists $w_i' \in \Lambda$ such that $gw_i' \in C_{gw_i}$. This implies that 
        \begin{align}
        \label{eq: w w 4}
            \|\pi_1(gw_i')\| \leq \|\pi_1(gw_i)\|, \quad \|\pi_2(gw_i')\| \leq \|\pi_2(gw_i)\|.
        \end{align}
       Also $w_i' \notin C_{w_i}$, so there exist $j \in \{1,2\}$ such that
        \begin{align}
            \label{eq: w w 6}
            \|\pi_j(w_i')\| > \|\pi_j(w_i)\|
        \end{align}
       Fix such a $j$. Note that 
        \begin{align}
        \label{eq: w w 5}
           | \|\pi_j(g w_i')\| - \|\pi_j( w_i')\|| \leq C_1 e\e, \quad | \|\pi_j(g w_i)\| - \|\pi_j( w_i)\|| \leq C_1 e \e.
        \end{align}
       Therefore we have
        \begin{align*}
            \left| \|\pi_j(w_i')\|- \|\pi_j(w_i)\| \right| &= \|\pi_j(w_i')\|- \|\pi_j(w_i)\| \quad \text{ using \eqref{eq: w w 6}} \\
            &\leq \|\pi_j(gw_i')\|- \|\pi_j(gw_i)\| + 2C_1 e \e \quad \text{ using \eqref{eq: w w 5}}\\
            &\leq 2C_1 e \e \quad \text{ using \eqref{eq: w w 4}.}
        \end{align*}

    Finally note that 
    \begin{align*}
            w_i' \in g^{-1} C_{gw_i}  \subset g^{-1}\{(x,y): \|x\| \leq 1, \|y\| < e\} \subset  \{(x,y): \|x\| \leq 1+ C_1\e e,  \|y\| < e+ C_1 e \e\}.
        \end{align*}

        Therefore we have $w_i, w_i' \in \{(x,y): \|x\| \leq 1+ C\e,  \|y\| < e+C\e\}$ such that for at least one $j \in \{1,2\}$, we have $|\|\pi_j(w_i)\|-\|\pi_j(w_i')\| |\leq C\e$. Hence $\Phi_{C\e}(w_i, w_i')=1$. This proves the observation. \vspace{0.5cm}

        Using Observations 1, 2, 3, 4 and the fact that $v_1, \ldots, v_p, w_1, \ldots, w_q$ are distinct, we get that 
        \begin{align*}
            &|f(g\Lambda)- f(\Lambda)| \leq p+q = (p-r)+ r + (q-s) + s \\
            &\leq \sum_{i=r+1}^p \varphi_{C\e}( v_i)+ \sum_{i=1}^r \Phi_{C\e}(v_i, v_i') + \sum_{i=s+1}^q \varphi_{C\e}( w_i)+  \sum_{i=1}^q \Phi_{C\e}(w_i, w_i')  \\ 
            & \leq \sum_{v\in \Lambda_\prim } \varphi_{C\e}(v)+   \sum_{\substack{v,w \in \Lambda_\prim \\ w \neq \pm v}} \Phi_{C\e}(v,w).
        \end{align*}
        Thus the lemma follows.
    \end{proof}

\section{Proof of Theorem \ref{thm: main effective levy general}}

The second main ingredient in the proof of Theorem \ref{thm: main effective levy general} is the following result.

\begin{thm}
\label{thm: main abstract theorem}
    Let $f$ be a bounded function on $\X$. Assume that there exists a family of non-negative measurable functions $\{\tau_\e\}_\e$ on $\X$ satisfying the following.
    \begin{itemize}
        \item There exists a constant $C$ such that 
        \begin{align}
            \label{eq: con: xi integral}
            \int_\X \tau_\e \, d\mu_\X \leq C \e.
        \end{align}
        \item For all small enough $\e>0$, and $g$ in $\e$-neighbourhood of identity in $G$, we have
    \begin{align}
        \label{eq:con: f cont}
        |f(g\Lambda)-f(\Lambda)| \leq \tau_\e(\Lambda) \quad \text{ for all } \Lambda \in \X.
    \end{align}
    \end{itemize}
    Then for any measure $\mu$ on $\Mat$ satisfying condition (EMEI), the following holds.
    \begin{enumerate}
        \item[(i)] For any $\varepsilon > 0$ and for $\mu$-almost every $\theta \in \Mat$, we have
        \begin{align}
            \label{eq:conclusion: effective}
            \sum_{s=0}^{N-1} f(a_{s}u(\theta)\Gamma) = N \gamma  + O_{ \e, \theta}(N^{1/2} \log^{\tfrac{1}{2}+\e} N),
        \end{align}
        where
        \begin{align}
            \label{eq: def gamma intro}
             \gamma= \mu_{\X}(f).
        \end{align}
        \item[(ii)] For $N \geq 1$, let us define $F_N: \X \rightarrow \R$ as
    \begin{equation}
        \label{eq: def F N}
        F_N(\Lambda) =\frac{1}{\sqrt{N}} \sum_{i=0}^{N-1} \left( f\circ a_i(\Lambda) - \gamma \right).
    \end{equation}
    Then for every $x \in \R$, we have
    \begin{align}
         \label{eq:conclusion: clt}
         \mu(\{\theta: F_N(u(\theta)\Gamma) < x \}) \rightarrow \mathrm{Norm}_{\sigma}(x) 
    \end{align}
    as $N \rightarrow \infty$, where
    \begin{align}
    \label{eq: def sigma}
         \sigma^2 =\sum_{s \in \Z} \left(  \int_{\X} f(a_s\Lambda) f(\Lambda)\, d\mu_\X(\Lambda) - \mu_{\X}(f)^2 \right)< \infty.
    \end{align}
    \end{enumerate}
\end{thm}

\begin{proof}[Proof of Theorem \ref{thm: main effective levy} assuming Theorem~\ref{thm: main abstract theorem}]
 Fix a measure $\mu$ on $\Mat$ satisfying condition (EMEI).  Using Lemma \ref{lem: bounded f}, \ref{lem: f is measurable} and \ref{lem: perturbed difference}, and the equations \eqref{eq: int varphi}, \eqref{eq: int PHI}, it is clear that $f$ satisfies the conditions of Theorem \ref{thm: main abstract theorem}. Therefore with $\gamma, \sigma, F_N$ as in statement of Theorem \ref{thm: main abstract theorem}, the equations \eqref{eq:conclusion: effective} and \eqref{eq:conclusion: clt} hold for $f$.

 Now using Lemmas \ref{lem: correpondence with best approx} and \ref{lem: bounded f}, we have for every $\e>0$ and for $\mu$-almost every $\theta$, the following holds.
 \begin{align*}
     \cN(\theta, T) &= \sum_{i=0}^{\lfloor T \rfloor -1} f(a_i u(\theta)\Gamma) + o(1) \\
     &= \gamma T + o( T^{1/2} \log^{\tfrac{3}{2} + \e } T) \quad \text{ using equation \eqref{eq:conclusion: effective}.}
 \end{align*}
 This proves part (i) of the theorem. For part (ii), let 
 $$
    \Delta_T(\theta) = \frac{\cN(\theta, T)- \gamma T}{T^{1/2}}.
    $$ 
    Then
    \begin{align}
    \label{eq: proof CLT 1}
    \Delta_T(\theta) = \alpha_T F_{\lfloor T \rfloor}(u(\theta)\Gamma) + \beta_T,
    \end{align}
    where 
    \begin{align}
    \label{eq: proof CLT 2}
        \alpha_T = \sqrt{\frac{\lfloor T \rfloor}{T}}  \longrightarrow 1 \quad \text{as }  T \rightarrow \infty,
    \end{align}
    and 
    \begin{align*}
      \beta_T &= \Delta_T(\theta) -\alpha_T F_{\lfloor T \rfloor} \\
      &= \frac{1}{T^{1/2}} \left( \left(\cN(\theta, T) - \sum_{s=0}^{\lfloor T \rfloor-1} f\circ a_s(u(\theta)\Gamma) \right)  -(T- \lfloor T \rfloor )\gamma   \right) \\
      &\ll \frac{1}{T^{1/2}} \longrightarrow 0.
    \end{align*}
    where the last inequality follows from Lemmas \ref{lem: correpondence with best approx} and \ref{lem: bounded f}. 
    Hence $\Delta_T(.)$ and $F_{\lfloor T \rfloor}(u(\theta)\Gamma)$ have the same distribution as $T \rightarrow \infty$. Part (ii) of the theorem now follows from equation \eqref{eq:conclusion: clt}.
\end{proof}

\section{Notation}
The rest of the paper is the paper is devoted to proving Theorem \ref{thm: main abstract theorem}. Henceforth, we fix a measure $\mu$ on $\Mat$ satisfying condition (EMEI). 

Denote by $U$ the subgroup
\begin{equation}
    U:= \left\{ u(\theta) : \theta \in M_{m \times n}(\R) \right\} < G.
\end{equation}
Let $\Y := U\Gamma = \{ u. \Gamma: u \in U\} \subset \X$. Geometrically, $\Y$ can be viewed as an $mn$-dimensional torus embedded in the space of lattices $\X$. We denote by $\mu_{\Y}$ the pusforward of $\mu$ under the map $\theta \mapsto u(\theta)\Gamma$.  We will need the following the result.\\

The first one follows directly from the fact that $\mu$ satisfies condition (EMEI).
\begin{cor}
 \label{mix hom}
    There exists $k \geq 1$ such that for $r \in \N$, there exists $ \delta_r' > 0$ satisfying the following: for all $\Psi_0 \in C_c^{\infty}(\Y )$, $\Psi_1, \ldots, \Psi_r \in C_c^{\infty}(\X)$ and $t_1, \ldots,t_r >0 $, we have 
    \begin{align}
    \label{eq: mix hom}
         \int_{\Y} \Psi_0(y) \left( \prod_{i=1}^r \Psi_i (a_{t_i} y)\right) \, d\mu_{\Y} = \left(\int_{\Y} \Psi_0 \, d\mu_{\Y}  \right) \prod_{i=1}^r \left( \int_{\X} \Psi_i \,d\mu_{\X} \right)  + \mathcal{O}_{ r} \left(e^{-\delta_r' D(t_1, \ldots, t_r)} \|\Psi_0\|_{C^k}  \prod_{i=1}^r \|\Psi_i\|_{C_k} \right),
    \end{align}
    where $D(t_1, \ldots, t_r) = \min\{t_i, |t_i - t_j|: 1\leq i ,j \leq r, i \neq j\}$
\end{cor}

The second is well known mixing property of $G$ action on $\X$.
\begin{thm}[{\cite[Cor.~3.5]{KM3}}]
    \label{thm: mixing}
    There exists $\delta_0>0$ and $k \geq 0$ such that for all  $\Psi_1, \Psi_2 \in C_c^k(\X)$ and $t \geq 0$, we have
   \begin{align}
   \label{eq: mixing}
       \int_\X \Psi_1(x) \Psi_2(a_tx) \, d\mu_\X(x) = \left(\int_{\X} \Psi_1 \, d\mu_\X(x) \right) \left(\int_{\X} \Psi_2 \, d\mu_\X(x) \right) + O(e^{-\delta_0 t} \|\Psi_1\|_{C^k} \|\Psi_2\|_{C^k}).
   \end{align}
\end{thm} 

For the remainder of the paper, we fix a natural number \(k\) and a constant \( \delta > 0 \) defined by
\begin{align}
\label{eq: def delta}
\delta = \min \{ \delta_1, \delta_2, \delta_1', \delta_2', \delta_0 \},
\end{align}
where \( \delta_r \) is as in \eqref{eq: mix hom identity},
\( \delta_r' \) is as in \eqref{eq: mix hom},
and \( \delta_0 \) is as in \eqref{eq: mixing}. \\

We also fix a function $f$ on $\X$ and a family of functions $(\tau_\e)_\e$, a constant $C$ satisfying equations \eqref{eq: con: xi integral} and  \eqref{eq:con: f cont}. Note that $f$ is not necessarily given by \eqref{eq: def f}, but any function satisfying conditions of Theorem \ref{thm: main abstract theorem}.  \\

 Also assume $M>0$ is such that $|f| \leq M$. We define
    \begin{align}
    \label{eq: def gamma}
    \gamma := \int_\X f\, d\mu_\X.
    \end{align}
    Further fix $\xi = \xi(m,n,k)>0$ be a constant such that 
 \begin{equation}
 \label{eq: def xi}
     \|h \circ a_t\| \ll e^{\xi t} \|h\|_{C^k}, \text{    for all $t \geq 0$ and } h \in C_c^{k}(\X), 
 \end{equation}
 where the suppressed constants are independent of $t$ and $h$. \\

\subsection{Convolution kernels}

To construct a smooth approximation of the function $f$, we will require mollifiers supported on small neighborhoods of the identity in $G$. The following well-known lemma provides such a family of smooth bump functions with uniform control on their $C^k$ -norms. 

\begin{lem}
\label{lem: convolution existence}
There exists a constant $C_\eta > 0$ and, for all sufficiently small $\e > 0$, a family of non-negative smooth functions $(\eta_\e)_{\e > 0}$ on $G$, supported on the $\e$-ball $B_\e^G$ around the identity in $G$, such that
\begin{align*}
    \int_G \eta_\e(g)\, dm_G(g) &= 1, \\
    \|\eta_\e\|_{C^j} &\leq C_\eta \e^{-(j + (m+n)^2 - 1)}.
\end{align*}
\end{lem}
\begin{proof}
Since $G$ is a real Lie group, there exists a neighbourhood $W$ of $0$ in the Lie algebra $\mathfrak g \simeq \R^{(m+n)^2-1}$ such that the exponential map
\[
\exp : W \to \exp(W) \subset G
\]
is a diffeomorphism onto a neighbourhood of the identity. Using this, it is enough to show the existence of mollifier functions on $W$, which is further homeomorphic to $B_1(0)$, the unit ball in $\R^{(m+n)^2-1}$. To prove the existence of mollifier functions on $B_1(0)$, we fix a non-negative smooth function $\eta$ on $\R^{(m+n)^2-1}$ with
\[
\supp (\eta) \subset B_1(0),
\qquad
\int_{\R^d} \eta(x)\,dx = 1,
\]
and define
\[
\eta_\e(x) := \e^{-d}\eta(x/\e).
\]
It is then clear that the function $\eta_\e$ satisfies the requisite properties, and hence the lemma follows.
\end{proof}

\noindent Throughout the paper, we fix $C_\eta$ and the family $(\eta_\e)$ as in Lemma \ref{lem: convolution existence}, and let 
\begin{align}
    \label{eq: def l}
    l := k + (m+n)^2 - 1.
\end{align}

\subsection{Cut-off function}
To construct a compactly supported smooth approximation of the function $\eta_\e *f$, we will require smooth cut-off functions on $\X$. The following lemma provides such a family of functions with uniform control on their $C^k$ -norms.

\begin{lem}[{\cite[Lem. 4.2, 4.11]{BG}}]
\label{lem: def phi}
There exists a constant $C_\phi > 0$ and, for all sufficiently small $\e > 0$, a family $(\phi_\e)$ in $C_c^\infty(\X)$ such that the following holds:
\begin{align*}
0\le \phi_\e \le 1,\quad \|\phi_\e\|_{C^k}\ll 1, \quad \int_{\X} \phi_\e \, d\mu_{\X} \geq 1- C_\phi \e.
\end{align*}
\end{lem}

\subsection{The function $f_\e$}
Throughout the paper, we define for all small enough $\e>0$
\begin{align}
\label{eq: def f epsilon}
    f_\e&= \phi_\e \cdot  \eta_\e *f.
\end{align}
Note that using Lemma \ref{lem: bounded f}, we have
\begin{align}
\label{eq: norm of tilde f}
    \|f_\e\|_{C^k} \ll \|\phi_\e\|_{C^k}  \| \eta_\e *f\|_{C^k} \ll \|\eta_\e\|_{C^k} \|f\|_{C^0} \ll \e^{-l}.
\end{align}

\section{Approximation on average}

\begin{lem}
    \label{lem: f is measurable}
    The function $f: \X \rightarrow \R$ is a measurable map.
\end{lem}
\begin{proof}
Let \( \epsilon_k = 2^{-k} \). Then
\[
\sum_{k=1}^{\infty} \int_\X \tau_{\epsilon_k} \, d\mu_\X \leq C \sum_{k=1}^\infty 2^{-k} < \infty.
\]
By Chebyshev's inequality and the Borel-Cantelli lemma, it follows that for \( \mu_\X \)-almost every \( \Lambda \in \X \),
\[
\tau_{\epsilon_k}(\Lambda) \to 0 \quad \text{as } n \to \infty.
\]

Now fix such a \( \Lambda \). For any sequence \( g_l \to e \) in \( G \),
\[
|f(g_l \Lambda) - f(\Lambda)| \leq \tau_{\delta_l}(\Lambda)
\]
where \( \delta_l = d(g_l, e) \). For large enough \( n \), \( \delta_l < \epsilon_{k_l} \) for increasing sequence some \( k_l \), so
\[
|f(g_n \Lambda) - f(\Lambda)| \to 0.
\]
Thus \( f(g \Lambda) \to f(\Lambda) \) as \( g \to e \). This shows that \( f \) is continuous at \( \mu_\X \)-almost every point along the \( G \)-orbit topology.

Since \( \X \) is standard Borel and \( G \) is second countable, continuity along orbits at almost every point implies that \( f \) coincides \( \mu_\X \)-a.e. with a Borel measurable function.

Therefore, \( f \) is \( \mu_\X \)-measurable.
\end{proof}

\begin{rem}
    Using Lemma \ref{lem: f is measurable}, we get that the convolution $\eta_\e *f$ makes sense, and hence the function $f_\e$ is well defined.
\end{rem}

\begin{lem}
    \label{lem: approximation f error}
    For all small enough $\e>0$ and all $t>0$, we have
    \begin{align}
        \label{eq: lem: approximation f error}
        \int_{\Mat} |f_\e - f|(a_t u(\theta) \Gamma) \, d\mu(\theta) \ll \e + e^{-\delta t} \e^{-l},
    \end{align}
    where $f_\e$ is defined as in \eqref{eq: def f epsilon}, $\delta$ is as in \eqref{eq: def delta}, and $l$ is as in \eqref{eq: def l}.
\end{lem}
\begin{proof}
    Let us define the function $\psi_\e: \X \rightarrow \R$ as 
    $$
    \psi_\e(\Lambda)= \sup_{g \in B_\e^G} \left(( |\eta_\e * f - f| \cdot \phi_\e) (g\Lambda)  \right).
    $$
    Then $\psi_\e$ is a positive measurable function, upper bounded by $2M$. Now consider the function $\eta_\e * \psi_\e$. First of all, this is a compactly supported smooth function, upper bounded by $M$ and $\|\eta_\e * \psi_\e\|_{C^k} \leq C_\eta M \e^{-l}$. Secondly, by the definition of $\psi_\e$, we have $\eta_\e* \psi_\e (\Lambda) \geq |\eta_\e * f - f|(\Lambda) \cdot \phi_\e(\Lambda)$. Therefore, we have
    \begin{align}
         &\int_{\Mat}  (|\eta_\e * f - f| \cdot \phi_\e) (a_t u(\theta)  \Gamma)  \, d\mu(\theta) \nonumber \\
         &\leq \int_{\Mat} \eta_\e* \psi_\e(a_t u(\theta) \Gamma) \, d\mu(\theta) \nonumber \\
         &= \int_\X \eta_\e* \psi_\e(\Lambda) \, d\mu_{\X}(\Lambda) + O(e^{-\delta t} \|\eta_\e* \psi_\e\|_{C^k}) \quad \text{ using \eqref{eq: mix hom identity}} \nonumber\\
         &= \int_\X \int_{G} \eta_{\e}(g_1) \psi_\e(g_1^{-1}\Lambda) \, d\mu_{\X}(\Lambda) dm_G(g) + O(e^{-\delta t} C_\eta M \e^{-l}). \label{eq:  a 1}
    \end{align}
    Note that for any $g_1$ in $\e$-ball around identity element of $G$, we have
    \begin{align}
        \psi_\e(g_1^{-1}\Lambda) &\leq \sup_{g_2 \in B_\e^G} |\eta_\e * f - f|(g_2 g_1^{-1}\Lambda)\nonumber \\
        &= \sup_{g_2 \in B_\e^G} \left| \int_{B_{\e}^G} \eta_\e(g_3) f(g_3^{-1} g_2 g_1^{-1}\Lambda) \, dm_G(g_3)- f(g_2 g_1^{-1}\Lambda) \right| \nonumber\\
        &\leq \sup_{g_2 \in B_\e^G}  \left( \int_{B_{\e}^G} \eta_\e(g_3) \left(  \left|   f(g_3^{-1} g_2 g_1^{-1}\Lambda)  - f(\Lambda) \right| + \left| f(g_2 g_1^{-1}\Lambda) -f(\Lambda) \right| \right) \, dm_G(g_3) \right) \nonumber\\
        &\leq 2\sup_{g \in B_{3\e}^G} |f(g\Lambda)-f(\Lambda)| \cdot \left(  \int_{B_{\e}^G} \eta_\e(g_3) \, dm_G(g_3) \right) =2\sup_{g \in B_{3\e}^G} |f(g\Lambda)-f(\Lambda)| \nonumber \\
        &\leq 2 \tau_{3\e}(\Lambda) \label{eq: a 2}
    \end{align}
     Combining \eqref{eq: a 2} with \eqref{eq:  a 1}, we get that
    \begin{align}
       \int_{\Mat} \left( |\eta_\e * f - f| \cdot \phi_\e \right)(a_t u(\theta) \Gamma) \, d\mu(\theta) \nonumber &\leq 2 \int_\X \tau_{3\e}(\Lambda)  \, d\mu_{\X}(\Lambda) + O(e^{-\delta t} C_\eta M \e^{-l}) \nonumber \\
       &\leq 6 C \e + O(e^{-\delta t} C_\eta M \e^{-l}) \nonumber \\
       &\ll \e + e^{-\delta t} \e^{-l}. \label{eq: a 3}
    \end{align}
    Also note that
    \begin{align}
        \int_{\Mat} (|f|\cdot (1- \phi_\e))(a_t u(\theta) \Gamma) \, d\mu(\theta) &\leq M \left( 1- \int_{\Mat} \phi_\e(a_t u(\theta) \Gamma)\, d\mu(\theta)  \right) \nonumber \\
        &\leq M \left( 1- \int_{\X} \phi_\e \, d\mu_\X + O(e^{-\delta t} \| \phi_\e\|_{C^k}  )  \right) \nonumber \\
        &\ll \e + e^{-\delta t}. \label{eq: a 4}
    \end{align}
    Finally, note that
    \begin{align}
        &\int_{\Mat} |f_\e - f|(a_t u(\theta) \Gamma) \, d\mu(\theta) \nonumber \\
        &\leq \int_{\Mat}  (|\eta_\e * f - f| \cdot \phi_\e) (a_t u(\theta)  \Gamma)  \, d\mu(\theta) + \int_{\Mat} (|f|\cdot (1- \phi_\e))(a_t u(\theta) \Gamma) \, d\mu(\theta). \label{eq: a 5}
    \end{align}
    
    The lemma now follows from \eqref{eq: a 3}, \eqref{eq: a 4} amd \eqref{eq: a 5}.
\end{proof}

\begin{lem}
\label{lem: integral fe X}
    For all $\e>0$, we have 
    $$
    \int_{\X} f_\e \, d\mu_\X = \gamma +  O( \e ),
    $$
    where $\gamma$ is defined as in \eqref{eq: def gamma}.
\end{lem}
\begin{proof}
    Note that
    \begin{align*}
        \int_{\X} f_\e \, d\mu_\X &= \int_\X \eta_\e*f \, d\mu_\X + \int_\X (\phi_\e-1) \cdot \eta_\e* f \cdot \, d\mu_\X \\
        &= \int_\X f\, d\mu_\X + O\left( 1- \int_\X \phi_\e \, d\mu_\X  \right) \quad \text{ using Lemma \ref{lem: bounded f} and invariance of $\mu_\X$} \\
        &= \gamma + O( \e ) \quad \text{ using Lemma \ref{lem: def phi} .}
    \end{align*}
    
\end{proof}

\begin{lem}
\label{lem: integral fe Y}
    For all $\e>0$ and $t \geq 0$, we have 
    $$
    \int_{\Mat} f_\e(a_tu(\theta)\Gamma) \, d\mu(\theta) = \gamma +  O( \e + e^{-\delta t} \e^{-l} ),
    $$
    where $\gamma$ is defined as in \eqref{eq: def gamma}, $\delta$ is as in \eqref{eq: def delta}, and $l$ is as in \eqref{eq: def l}.
\end{lem}
\begin{proof}
    Note that 
    \begin{align*}
        \int_{\Mat} f_\e(a_tu(\theta)\Gamma) \, d\mu(\theta)  &= \int_\X f_\e \, d\mu_\X + O\left( e^{-\delta t} \|f_\e\|_{C^k} \right) \text{ using \eqref{eq: mix hom identity}}\\
        &= \gamma +  O( \e ) + O(e^{-\delta t} \e^{-l}) \quad \text{ using Lemma \ref{lem: integral fe X} and \eqref{eq: norm of tilde f}.}
    \end{align*}
    
\end{proof}

\begin{lem}
    \label{lem: birkhoff average}
    For $t \geq 0$, the following holds
    $$ \int_{\Mat} f (a_tu(\theta)\Gamma) \, d\mu(\theta) = \gamma  + O(e^{-\frac{\delta t}{l+1}}),$$
    where $\gamma$ is defined as in \eqref{eq: def gamma}, $\delta$ is as in \eqref{eq: def delta}, and $l$ is as in \eqref{eq: def l}.
\end{lem}
\begin{proof}
    Let $\e= e^{-\frac{\delta t}{l+1}}$. Note that
    \begin{align*}
        \int_{\Mat} f (a_tu(\theta)\Gamma) \, d\mu(\theta)   &= \int_{\Mat} f_\e (a_tu(\theta)\Gamma) \, d\mu(\theta)  + O \left(\int_{\Mat} \left| f_\e - f\right| (a^t u(\theta)\Gamma) \, d\mu(\theta) \right) \\
        &= \gamma +  O( \e + e^{-\delta t} \e^{-l} )  \quad \text{ using Lemma \ref{lem: approximation f error} and \ref{lem: integral fe Y}} \\
        &= \gamma  + O(e^{-\frac{\delta t}{l+1}}).
    \end{align*}
    
\end{proof}

\begin{lem}
    \label{lem: var 3}
    For $s \geq 0$ and $\e>0$, we have
    \begin{align*}
     \int_{\X}f_\e(a_s\Lambda)f_\e(\Lambda) \, d\mu_{\X}(\Lambda) - \mu_{\X}(f_\e)^2 = \int_{\X} f(a_s\Lambda) f(\Lambda)\, d\mu_\X(\Lambda)- \mu_{\X}(f)^2 + O(\e).
    \end{align*}
\end{lem}
\begin{proof}
    Note that 
    \begin{align*}
        &\left| \int_{\X}f_\e(a_s\Lambda)f_\e(\Lambda) \, d\mu_{\X}(\Lambda)- \mu_{\X}(f_\e)^2  -\int_{\X} f(a_s\Lambda) f(\Lambda)\, d\mu_\X(\Lambda) + \mu_{\X}(f)^2 \right| \\
        &\leq 2(\|f_\e\|_{C^0} + \|f\|_{C^0})\int_{\X} \left| f_\e -f\right| \, d\mu_\X \ll \int_{\X} \left| f_\e -f\right| \, d\mu_\X  \\
        &\ll \|\eta_\e *f\|_{C^0}\int_{\X} \left| 1- \phi_\e \right|  \, d\mu_\X +  \int_{\X}   \left| \eta_\e *f -f\right| \, d\mu_\X  \\
        &\ll \e +  \int_\X  \sup_{g \in B_\e^G}|f(g\Lambda) -f(\Lambda)| \, d\mu_\X(\Lambda) \\
        &\leq \e + \int_\X \tau_\e \, d\mu_\X(\Lambda) \quad \text{ using Lemma \ref{lem: perturbed difference}}\\
        &\ll \e + \e .
    \end{align*}
    This proves the lemma.
\end{proof}

\begin{lem}
    \label{lem: var 4}
    For $s \geq 0$, we have
    \begin{align}
    \int_{\X} f(a_s\Lambda) f(\Lambda)\, d\mu_\X(\Lambda) - \mu_{\X}(f)^2 = O \left(e^{\frac{-\delta s}{2l+1}} \right),
    \end{align}
    where $\delta$ is defined as in \eqref{eq: def delta}, and $l$ is as in \eqref{eq: def l}.
\end{lem}
\begin{proof}
    Let $\e= e^{\frac{-\delta s}{2l+1}}$. Note that
    \begin{align*}
        &\int_{\X} f(a_s\Lambda) f(\Lambda)\, d\mu_\X(\Lambda) - \mu_{\X}(f)^2 \\
        &= \int_{\X}f_\e(a_s\Lambda)f_\e(\Lambda) \, d\mu_{\X}(\Lambda)- \mu_{\X}(f_\e)^2 + O(\e) \quad \text{ using Lemma \ref{lem: var 3}}  \\
        &= O(e^{-\delta s}\|f_\e\|_{C^k}^2) + O(\e) \quad \text{ using Theorem \ref{thm: mixing}}\\
        &=O(e^{-\delta s} \e^{-2l} + \e) = O \left(e^{\frac{-\delta s}{2l+1}} \right).
    \end{align*}
    
\end{proof}

\section{Effective convergence of Birkhoff sums}

To derive estimates on Birkhoff sums, i.e., the first part of Theorem \ref{thm: main abstract theorem}, we will need the following theorem from Kleinbock-Shi-Weiss \cite{KSW} (see also \cite[Thm.~A.1]{aggarwalghosh2024joint}).
\begin{thm}[{\cite[Thm.~3.1]{KSW}}]
    \label{thm: KSW}
    Let $(Y,\nu)$ be a probability space and $F: Y \times \R_{\geq 0} \rightarrow \R$ be a bounded measurable function. Suppose that there exists $\delta_Y>0$ and $C_Y>0$ such that for any $s \geq t \geq 0$, ,
    \begin{align*}
        \left|\int_Y F(y,t) F(y,s) \, d\nu(y) \right| \leq C_Y e^{-\delta_Y \min\{t, s-t\}}.
    \end{align*}
    Then given $\e>0$ we have 
    \begin{align*}
        \frac{1}{T}\int_0^T F(y,t) \, dt = o(T^{-1/2} \log^{\frac{3}{2}+ \e} T),
    \end{align*}
    for $\nu$-almost every $y \in Y$.
\end{thm}

The above theorem will be used in the form of the following corollary to derive the first part of Theorem \ref{thm: main abstract theorem}.
\begin{cor}
    \label{cor: KSW}
    Let $(Y,\nu)$ be a probability space and $F: Y \times \R_{\geq 0} \rightarrow \R$ be a bounded measurable function. Suppose that there exists $\delta_Y>0$ and $C_Y>0$ such that for any $s \geq t \geq 0$, 
    \begin{align}
        \label{eq: KSW 1'}
        \left|\int_Y F(y,t) F(y,s) \, d\nu(y) \right| \leq C_Y e^{-\delta_Y \min\{t, s-t\}}.
    \end{align}
    Then given $\e>0$ we have 
    \begin{align}
        \label{eq: KSW 2'}
        \frac{1}{N} \sum_{i=1}^{N-1} F(y,i) = o(N^{-1/2} \log^{\frac{3}{2}+ \e}N ),
    \end{align}
    for $\nu$-almost every $y \in Y$.
\end{cor}
\begin{proof}
    Let us define $\tilde{F}: Y \times \R_{\geq 0} \rightarrow \R$ as $\tilde{F}(y,t)= F(y, \lfloor t \rfloor)$, where $\lfloor t \rfloor$ equals the greatest integer less than or equal to $t$. Then for any $s \geq t \geq 0$, 
    \begin{align*}
        \left|\int_Y \tilde{F}(y,t) \tilde{F}(y,s) \, d\nu(y) \right|  &= \left|\int_Y F(y,\lfloor t \rfloor) F(y,\lfloor s \rfloor) \, d\nu(y) \right| \\
        &\leq C_Y e^{-\delta_Y \min\{\lfloor t \rfloor, \lfloor s \rfloor-\lfloor t \rfloor\}} \\
        &\leq C_Y e^2 e^{-\delta_Y \min\{t, s-t\}}.
    \end{align*}
    Therefore using Theorem \ref{thm: KSW}, we get that for any $\e>0$
    \begin{align}
        \label{eq: b 1}
        \frac{1}{T}\int_T \tilde{F}(y,t) \, dt = o(T^{1/2} \log^{\frac{3}{2}+ \e} T),
    \end{align}
    for $\nu$-almost every $y \in Y$. Putting $T= N$ and using the definition of $\tilde{F}$ in \eqref{eq: b 1}, we get \eqref{eq: KSW 2'}. Thus the corollary follows.
\end{proof}

The following lemma proves the estimate in \eqref{eq: KSW 1'} for the function $f$, enabling the use of Corollary \ref{cor: KSW}.

\begin{lem}
    \label{lem: l2 estimates}
    There exists constants $C_f>0$ such that for any $s \geq t \geq 0$, 
    \begin{align}
        \label{eq: lem l2 estimate}
        \int_{\Mat} (f(a_t u(\theta) \Gamma)-\gamma) (f(a_s u(\theta) \Gamma) -\gamma) \, d\mu(\theta) \leq C_f e^{-\frac{\delta}{2l+1} \min\{t, s-t\}},
    \end{align}
    where $\gamma $ is defined as in \eqref{eq: def gamma}, $\delta$ is as in \eqref{eq: def delta}, and $l$ is as in \eqref{eq: def l}. 
\end{lem}
\begin{proof}
Set 
\[
\varepsilon \;=\; \exp\!\Bigl(-\tfrac{\delta}{2l+1}\,\min\{t,\;s-t\}\Bigr).
\]
Then decompose
\[
\int_{\Mat}\bigl(f(a_tu(\theta)\Gamma)-\gamma\bigr)\bigl(f(a_su(\theta)\Gamma)-\gamma\bigr)\,d\mu(\theta)
\;=\; I_1 \;+\; I_2 \;+\; I_3,
\]
where
\begin{align}
I_1 &= \int_{\Mat}\bigl(f_\e(a_tu(\theta)\Gamma)-\gamma\bigr)
\bigl(f_\e(a_su(\theta)\Gamma)-\gamma\bigr)\,d\mu(\theta),\\
I_2 &= \int_{\Mat}\bigl(f(a_tu(\theta)\Gamma)-f_\e(a_tu(\theta)\Gamma)\bigr)
\bigl(f_\e(a_su(\theta)\Gamma)-\gamma\bigr)\,d\mu(\theta),\\
I_3 &= \int_{\Mat}\bigl(f(a_tu(\theta)\Gamma)-\gamma\bigr)
\bigl(f(a_su(\theta)\Gamma)-f_\e(a_su(\theta)\Gamma)\bigr)\,d\mu(\theta).
\end{align}
Since $|f|\le M$ by Lemma \ref{lem: bounded f}, we get
\[
|I_2|\;\le\;2M\int_{\Mat}\bigl|f-f_\e\bigr|(a_tu(\theta)\Gamma)\,d\mu(\theta),
\quad
|I_3|\;\le\;2M\int_{\Mat}\bigl|f-f_\e \bigr|(a_su(\theta)\Gamma)\,d\mu(\theta).
\]
By Lemma \ref{lem: approximation f error}, each of these is less than $O\bigl(\varepsilon + e^{-\delta t}\,\varepsilon^{-l}\bigr) $, which is further less than $ O(e^{-\frac{\delta}{2l+1} \min\{t, s-t\}})$. Next, we have
\begin{align*}
&I_1 =  \int_{\Mat} f_\e(a_tu(\theta)\Gamma) f_\e(a_su(\theta)\Gamma)\,d\mu(\theta) +\gamma^2- \gamma \left(  \int_{\Mat} f_\e(a_tu(\theta)\Gamma) \,d\mu(\theta) + \int_{\Mat} f_\e(a_su(\theta)\Gamma)\,d\mu(\theta) \right)  \nonumber \\
&= \mu_\X(f_\e)^2 + O(e^{-\delta \min\{s-t,t\}} \|f_\e\|_{C^k}^2) +\gamma^2 - \gamma \left( 2\gamma + O(\e + e^{-\delta t} \e^{-l}) \right) \quad  \text{using \eqref{eq: mix hom identity}, Lemma \ref{lem: integral fe Y}} \\
&= O(\e + e^{-\delta \min\{s-t,t\}} \e^{-2l})  \quad \text{ using Lemma \ref{lem: integral fe X} and \eqref{eq: norm of tilde f}}\\
&= O(e^{-\frac{\delta}{2l+1} \min\{t, s-t\}}).
\end{align*}

Combining the bounds for \(I_1,I_2,I_3\), the lemma follows.
\end{proof}

\begin{proof}[Proof of Theorem \ref{thm: main abstract theorem}(i)]
    The first part of Theorem \ref{thm: main abstract theorem} follows directly from Corollary \ref{cor: KSW} and Lemma \ref{lem: l2 estimates}.
\end{proof}

\section{Central limit theorem}
This section aims to prove the second part of Theorem \ref{thm: main abstract theorem}.

\subsection{Approximation by compactly supported continuous functions}
\label{subsec: approximation}
To prove the second part of Theorem \ref{thm: main abstract theorem}, we define
\begin{align}
    \label{eq: tF N}
    \tF_N = \frac{1}{\sqrt{N}} \sum_{i=K}^{N-1} \left( f_\e \circ a_i(\Lambda) - \mu_\Y(f_\e\circ a_i) \right),
\end{align}
for $K $ equals the least integer greater than or equal to $N^{1/2} (\log N)^{-1}$, and $\e= N^{-100}$. Note that 
\begin{align}
    \|F_N - \tF_N\|_{L^1(\Y)} &\leq \frac{1}{\sqrt{N}} \sum_{i < K} \left\|f\circ a_i - \gamma \right\|_{L^1(\Y)} + \frac{1}{\sqrt{N}} \sum_{i \geq K}  \int_{\Y} |f_\e-f|\circ a_i \, d\mu_{\Y} \nonumber \\
     &\quad + \frac{1}{\sqrt{N}}\sum_{i>K} |\mu_{\Y}(f_\e \circ a_i) - \gamma| \nonumber\\
    &\ll \frac{K}{\sqrt{N}} +  \sqrt{N} \left( \e + e^{-\delta K} \e^{-l} \right)   \longrightarrow 0, \label{eq: approx is good}
\end{align}
where in the last inequality we have used Lemmas \ref{lem: bounded f}, \ref{lem: approximation f error}, \ref{lem: integral fe Y} and the definition of $\mu_{\Y}$. Equation \eqref{eq: approx is good} implies that $F_N$ and $\tF_N$ will have the same convergence in distribution as $N \rightarrow \infty$.  Hence, if we can prove the CLT for $(\tF_N)$, then the CLT for $(F_N)$ would follow.

\subsection{Method of Cumulants}
Let $(\z, \nu)$ be a probability space. Given bounded measurable functions $\Psi_1, \ldots , \Psi_r$ on $\z$, we define their \textit{joint cumulant} as $$\cum_{\nu}^{(r)}(\Psi_1, \ldots, \Psi_r) = \sum_{\Pp} (-1)^{|\Pp|-1} (|\Pp|-1)! \prod_{I \in \Pp} \int_{\z} \left( \prod_{i \in I} \Psi_i \right) \, d\nu, $$ where the sum is taken over all partitions $\Pp $ of the set $ \{1, \ldots , r\}.$ For a bounded measurable function $\Psi$ on $\z$, we also set $$\cum^{(r)}_\nu(\Psi) = \cum^{(r)}_\nu (\Psi, \ldots, \Psi).$$ We have the following classical CLT-criterion (see, for instance, \cite{FS}, \cite[Prop.~5.1]{BG2}).

\begin{prop}
\label{Method of Cumulants}
 Let $(G_T)$ be a sequence of real-valued bounded measurable functions such that $$\int_{\z} G_T \, d\nu=0 \ \text{ and } \ \sigma^2:= \lim_{T \rightarrow \infty} G_T^2 \, d\nu < \infty$$ and $$\lim_{T \rightarrow \infty} \cum^{(r)}_\nu(G_T) = 0, \text{   for all } r \geq 3.$$ Then, for every $x \in \R$, $$\nu(\{ G_T < x \}) \rightarrow \mathrm{Norm}_{\sigma}(x) \text{ as } T \rightarrow \infty.$$ 
    
\end{prop}

We will prove a CLT for $\tF_N$ with the help of Proposition \ref{Method of Cumulants}, with $\z= \Y$, $\nu = \mu_{\Y}$.

\subsection{Estimating Cumulants}
\label{subsec: Cumulants}
Fix $r > 2$. Using the computation in \cite[Section~3.2]{BG}, and replacing Corollary~\ref{mix hom} in place of \cite[Corollary~3.3]{BG} in the relevant steps, equation (3.25) of \cite{BG} yields
\begin{align}
\label{eq:cum evaluation}
    \left|\cum^{(r)}_{\mu_\Y}(\tF_N) \right| &\ll \left( (\beta_{r+1} +1)^rN^{-r/2} + \left(\max_j \alpha_j^{r-1} \right) N^{1-r/2} \right) \|f_\e\|^r_{C^0} \\
    &+ N^{r/2} \left( \max_j e^{-(\delta_r' \beta_{j+1} - r \xi \alpha_j)} \right) \|f_\e\|^r_{C^k},
\end{align}
where $\delta_r'$ is defined as in Corollary~\ref{mix hom} and $\alpha_i, \beta_j$ are any real numbers satisfying
\begin{align}
\label{eq: cumulants 1}
    0= \alpha_0 <\beta_1 < \alpha_1 =(3+r)\beta_1 &< \beta_2 < \cdot < \beta_r< \alpha_r= (3+r)\beta_r < \beta_{r+1}, \\
\label{eq: cumulants 2}
   \delta_r' \beta_{j+1} - r \xi \alpha_j &>0, \quad \text{ for } j=1, \ldots, r.
\end{align}
First of all note that
\begin{align}
    \label{eq: cumulants 3}
    \|f_\e\|^r_{C^0} &\leq M^r \ll 1 \quad \text{using Lemma \ref{lem: bounded f},} \\
    \label{eq: cumulants 4}
    \|f_\e\|^r_{C^k} &\ll  \e^{-lr} \quad \text{using \eqref{eq: norm of tilde f}}.
\end{align}
Now let $\kappa = N^{1/16}$, and define the parameters $\beta_j$ inductively by the formula
\begin{align}
    \label{eq:cumulants 5}
    \beta_{j+1}= \max\left\{\kappa + (3+r)\beta_j , \kappa+ (\delta_r')^{-1}r  (3+r)\xi \beta_j \right\}.
\end{align}
Clearly then \eqref{eq: cumulants 1} and \eqref{eq: cumulants 2} are satisfied. Also, it easily follows by induction that $\beta_{r+1} \ll_r \kappa$, and we deduce from \eqref{eq:cum evaluation} and \eqref{eq: cumulants 3}, \eqref{eq: cumulants 4} that
\begin{align*}
    \left|\cum^{(r)}_{\mu_\Y}(\tF_N) \right| &\ll (\kappa+1)^{r}N^{-r/2} + \kappa^{r-1} N^{1-r/2} + N^{r/2}e^{-\delta_r' \kappa} \e^{-lr}.
\end{align*}
Using fact that $\e= N^{-100}$ and $\kappa = N^{1/16}$, we conclude for $r >2$  that
\begin{align}
    \label{eq: cum equal zero}
    \lim_{N \rightarrow \infty} \cum^{(r)}_{\mu_\Y}(\tF_N)  = 0.
\end{align}

\subsection{Estimating variance}

For notational simplicity, we define
\begin{align}
    \label{eq: def psi}
    \psi_{s}^{(N)}(\Lambda)&= f_\e\circ a_s(\Lambda) - \mu_\Y(f_\e\circ a_s), \quad \text{for $\e= 1/N^{100}$}
\end{align}

We will need the following lemmas for computing the variance of $\tilde{F}_N$.

\begin{lem}
    \label{lem: var 1}
    For $s, t \geq 0$ and $N \in \N$, we have
    \begin{align}
        \int_{\Y} \psi_{s+t}^{(N)}(\Lambda) \psi_{t}^{(N)}(\Lambda) \, d\mu_{\Y}(\Lambda) = O(e^{-\delta \min\{s,t\}} N^{200 l} ),
    \end{align}
    where $\delta$ is defined as in \eqref{eq: def delta}, and $l$ is as in \eqref{eq: def l}.
\end{lem}
\begin{proof}
    Using Corollary \ref{mix hom} and \eqref{eq: norm of tilde f}, we get that for $i=s+t,t$
    \begin{align}
        \mu_\Y(f_\e\circ a_{i}) &= \mu_{\X}(f_\e) + O(e^{-\delta t} \|f_\e\|_{C^k})= \mu_{\X}(f_\e) + O(e^{-\delta t} \e^{-l} ). \label{eq: lem var 1 1}
    \end{align}
    Corollary \ref{mix hom} and \eqref{eq: norm of tilde f} also gives that
    \begin{align}
        \int_\Y f_\e(a_{s+t} \Lambda) f_\e(a_{t}\Lambda) \, d\mu_{\Y}(\Lambda) &= \mu_{\X}(f_\e)^2 + O(e^{-\delta \min\{s,t\}} \|f_\e\|_{C^k}^2) \nonumber \\
        &= \mu_{\X}(f_\e)^2 + O(e^{-\delta\min\{s,t\}} \e^{-2l}). \label{eq: lem var 1 2}
    \end{align}
   Since
    \begin{align}
    \label{eq: lem var 1 3}
        \int_{\Y} \psi_{s+t}^{(N)}(\Lambda) \psi_{t}^{(N)}(\Lambda) \, d\mu_{\Y}(\Lambda) = \int_\Y f_\e(a_{s+t} \Lambda)  f_\e(a_{t}\Lambda) \, d\mu_{\Y}(\Lambda)  - \mu_\Y(f_\e\circ a_{s+t}) \cdot\mu_\Y(f_\e\circ a_{t}),
    \end{align}
    the lemma follows from \eqref{eq: lem var 1 1} and \eqref{eq: lem var 1 2} and the fact that $\e = N^{-100}$.
\end{proof}

\begin{lem}
    \label{lem: var 2}
    For $s, t \geq 0$ and $N \in \N$, we have
    \begin{align*}
        \int_{\Y} \psi_{s+t}^{(N)}(\Lambda) \psi_{t}^{(N)}(\Lambda) \, d\mu_{\Y}(\Lambda)= \int_{\X}f_\e(a_s\Lambda)f_\e(\Lambda) \, d\mu_{\X}(\Lambda) - \mu_{\X}(f_\e)^2 +  O(e^{-\delta t+ \xi s} N^{200 l} ),
    \end{align*}
    where $\e= N^{-100}$, $\delta$ is defined as in \eqref{eq: def delta}, and $l$ is as in \eqref{eq: def l}.
\end{lem}
\begin{proof}
    Using Corollary \ref{mix hom}, we have
    \begin{align}
        \int_\Y f_\e (a_{s+t} \Lambda)  f_\e(a_{t}\Lambda) \, d\mu_{\Y}(\Lambda)  &= \int_{\X} f_\e (a_{s} \Lambda)  f_\e(\Lambda) \, d\mu_{\X}(\Lambda) + O(e^{-\delta t} \|(f_\e \circ a_s )f_\e\|_{C^k}) \nonumber\\
        &= \int_{\X} f_\e (a_{s} \Lambda)  f_\e(\Lambda) \, d\mu_{\X}(\Lambda) + O(e^{-\delta t}  e^{\xi s} \|f_\e\|_{C^k}^2). \label{eq: var 2 1}
    \end{align}
    The lemma now follows from \eqref{eq: norm of tilde f}, \eqref{eq: lem var 1 1}, \eqref{eq: lem var 1 3} and \eqref{eq: var 2 1}. 
\end{proof}

We now compute the variance.
\begin{lem}
    \label{lem:varaince estimate}
    The following holds
    \begin{align}
        \lim_{N \rightarrow \infty} \int_{\Y} \tilde{F}_N^2(\Lambda)\, d\mu_{\Y}(\Lambda) = \Xi(0) + 2\sum_{s=1}^\infty \Xi(s) = \sum_{s \in \Z} \Xi(s) < \infty,
    \end{align}
    where 
    $$
    \Xi(s)= \int_{\X} f(a_s\Lambda) f(\Lambda)\, d\mu_\X(\Lambda) - \mu_{\X}(f)^2.
    $$
\end{lem}
\begin{proof}
    Note that
    \begin{align}
        \int_{\Y} \tilde{F}_N^2(\Lambda)\, d\mu_{\Y}(\Lambda) &= \frac{1}{N} \sum_{i,j=K}^{N-1} \int_{\Y} \psi_{i}^{(N)} \psi_{j}^{(N)} \, d\mu_\Y \nonumber \\
        &= \frac{1}{N} \sum_{t=K}^{N-1} \int_{\Y} (\psi_{t}^{(N)})^2 \, d\mu_\Y + 2 \sum_{s=1}^{N-K-1} \left(  \frac{1}{N} \sum_{t=K}^{N-1-s} \int_{\Y} \psi_{s+t}^{(N)} \psi_{t}^{(N)} \, d\mu_\Y \right). \label{var comp 1}
    \end{align}
    Now using Lemma \ref{lem: var 1}, we have for $t \geq K$ and $s \geq \delta K/(2\xi)$
    \begin{align}
        \frac{1}{N} \sum_{t=K}^{N-1-s} \int_{\Y} \psi_{s+t}^{(N)} \psi_{t}^{(N)} \, d\mu_\Y \ll e^{-\alpha K} N^{200 l}, \label{var comp 2}
    \end{align}
    where $\alpha = \min\{ \delta/2,\delta^2/(2\xi) \}$. Using Lemma \ref{lem: var 2} and \ref{lem: var 3}, we get for $s< \delta K/(2\xi)$ and $t \geq K$
    \begin{align}
        \int_{\Y} \psi_{s+t}^{(N)} \psi_{t}^{(N)} \, d\mu_\Y =\Xi(s) +O\left( \frac{1}{N^{100}} + e^{-\alpha K} N^{200 l} \right). \label{var comp 3}
    \end{align}
    Splitting the sum in \eqref{var comp 1} for $s \geq \delta K/(2\xi)$ and $s< \delta K/(2\xi)$, and using \eqref{var comp 2}, \eqref{var comp 3}, we get that
     \begin{align}
        \int_{\Y} \tilde{F}_N^2(\Lambda)\, d\mu_{\Y}(\Lambda) &= \frac{N-(K+1)}{N}\  \Xi(0) +  2\sum_{\substack{1 \leq s < \delta K/(2\xi)}} \frac{N-(K+s+1)}{N}\  \Xi(s) \nonumber \\
        &+ O\left( \frac{K}{N^{100}} + K e^{-\alpha K} N^{200 l} +e^{-\alpha K} N^{200l + 1}  \right). \label{var comp 4}
    \end{align}
    Note that using Lemma \ref{lem: var 4}, we have $\Xi(s) \ll e^{\frac{-\delta s}{2l+1}} $ for $s \geq 0$. Therefore by the dominated convergence theorem, we have
    \begin{align}
        \lim_{N \rightarrow \infty} \sum_{\substack{1 \leq s <  \delta K/(2\xi)}} \frac{N-(K+s+1)}{N}\  \Xi(s) = \sum_{s=1}^\infty \Xi(s) <\infty. \label{var comp 5}
    \end{align}
    The lemma now follows from \eqref{var comp 4} and \eqref{var comp 5}.
\end{proof}

\subsection{Proof of Theorem \ref{thm: main abstract theorem}(ii)}
\begin{proof}[Proof of Theorem \ref{thm: main abstract theorem}(ii)]
    Using Section \ref{subsec: approximation} and Theorem \ref{Method of Cumulants}, the second part of Theorem \ref{thm: main abstract theorem} holds if and only if
    \begin{align*}
        \int_{\Y} \tilde{F}_N \, d\mu_{\Y} &=0, \\
        \lim_{N \rightarrow \infty} \int_\Y \tilde{F}_N^2 \,d\mu_\Y &= \sum_{s\in \Z}^\infty \Xi(s)\\
        \lim_{N \rightarrow \infty} \cum^{(r)}(\tilde{F}_N) &= 0, \text{   for all } r \geq 3.
    \end{align*}
    The above equations holds by Section \ref{subsec: Cumulants} and Lemma \ref{lem:varaince estimate}. Hence the theorem holds.
\end{proof}

\section*{Statements and Declarations}

\noindent\textbf{Funding.} A.\ G.\ gratefully acknowledges support from a grant from the Infosys Foundation to the Infosys Chandrasekharan Random Geometry Centre. 
G.\ A.\ and A.\ G.\ gratefully acknowledge a grant from the Department of Atomic Energy, Government of India, under project 12-R\&D-TFR-5.01-0500. 
This work was supported by a Royal Society International Exchanges Grant. 
Part of the work was done when G.\ A.\ was visiting the University of Bristol on a visit funded by this grant; the support of the grant and the hospitality of the University of Bristol are gratefully acknowledged. \\

\noindent\textbf{Competing Interests.} The authors declare no competing interests. \\

\noindent\textbf{Data Availability.} Data sharing not applicable to this article as no datasets were generated or analyzed. \\

\noindent\textbf{Acknowledgements.} G.\ A.\ thanks the Hausdorff Center for Mathematics for its hospitality and M.\ Einsiedler for helpful conversations. 
We thank Yann Bugeaud, Jens Marklof, and Subhajit Goswami for helpful discussions. 
The question of proving an effective version of the L\'evy–Khintchine theorem and central limit theorems was raised by Sary Drappeau during a talk by A.\ G.\ on our joint work on Cheung’s conjecture at the Annual Discussion Meeting in Analytic Number Theory, ISI Kolkata, February 2025. 
We are grateful to Sary for posing the question and to the organizers for the invitation and hospitality. The authors would also like to thank the anonymous referee for a very careful reading and comments which have improved the clarity of the paper.

\bibliography{Biblio}
\end{document}